\documentclass{article}
\usepackage[T1]{fontenc}
\usepackage{graphicx}
\usepackage{amsthm,amsmath,amssymb}
\usepackage[colorlinks=true,citecolor=black,linkcolor=black,urlcolor=blue]{hyperref}

\theoremstyle{plain}
\newtheorem{thm}{Theorem}
\newtheorem{lem}[thm]{Lemma}
\newtheorem{cor}[thm]{Corollary}
\newtheorem{prop}[thm]{Proposition}

\theoremstyle{definition}

\newtheorem{exmp}[thm]{Example}
\newtheorem{conj}[thm]{Conjecture}

\theoremstyle{remark}
\newtheorem{rem}[thm]{Remark}

\title{On the abelianity of the stochastic sandpile model.}
\author{Fran\c cois Nunzi\\
\small IRIF and Universit\'e Paris Diderot, Case 7014\\
\small F-75205 Paris Cedex 13\\
\small France\\
\small \tt fnunzi@liafa.univ-paris-diderot.fr}

\date{\today}

\begin{document}

\maketitle

\begin{abstract}
We consider a stochastic variant of the Abelian Sandpile Model (ASM) on a finite graph, introduced by Chan, Marckert and Selig. 
Even though it is a more general model, some nice properties still hold. We show that on a certain probability space, even if we lose the group structure due to topplings not being deterministic, some operators still commute. 
As a corollary, we show that the stationary distribution still does not depend on how sand grains are added onto the graph in our model, answering a conjecture of Selig.

\end{abstract}


\section{Introduction}

The abelian sandpile model (ASM in the rest of the paper) was introduced by Bak, Tang and Wiesenfeld \cite{BTW}, as an example of a dynamical system displaying self-organized criticality.
A very similar model, known as the chip-firing game, had been introduced by Spencer \cite{Spencer} a little earlier : we consider a pile of $N$ chips in the center of a long path, and at each step, the player moves $\lfloor\frac{N}{2}\rfloor$ to the left and $\lfloor\frac{N}{2}\rfloor$ to the right. He then does the same thing with the two new piles obtained, in whatever order, and so on, until no site contains more than one chip. This game was then refined, allowing the player to move only two chips (one to the left and one to the right) at each step, and starting with any configuration, and then a generalization to simple graphs was studied by Bj\"{o}rner, Lov\'asz and Shor in \cite{BLS}.
The ASM is a Markov chain and can be seen as a succession of chipfiring games on a graph, the states being configurations where no vertex can chipfire : at each step, we add a chip (or sand grain) on one of the vertices of the graph, we then proceed as if we were playing a chipfiring game, until no vertex can chipfire (or topple), thus getting to our new state.
Many papers are dedicated to the mathematical aspects of the ASM, and we will just cite Redig's \cite{Redig} and Dhar's review \cite{Dhar}. A first stochastic extension, known as the Manna model, introduced in \cite{Manna}, and studied further by Dahr \cite{Dhar99}, lead to an extension of ASM, called the stochastic sandpile model (SSM from now on), studied by Chan, Marckert and Selig \cite{CMS}. In this model, when a vertex topples, it only gives sand grains to some of its neighbors, in such a way that every partial toppling may occur. The authors of \cite{CMS} studied the recurrent states, and the Lacking polynomial which is very similar to the Tutte polynomial.
In his thesis \cite{Selig}, Thomas Selig stated a conjecture : 
\begin{conj}\label{Conj}
	Let $G=(V,E)$ be a graph and $\mu$ be a probability distribution on $V$. Then the stationary distribution for the SSM where grain additions are made according to the distribution $\mu$ does not depend on $\mu$.
\end{conj}
The aim of our paper is to prove this conjecture and generalize it to the case where only a subset of the set of partial topplings my occur. 
In order to do that, we will need to describe the Markov chain as some sort of game, with decks of cards on each site, governing the way those sites will topple. The deck presentation was used by Diaconis and Fulton in \cite{DF}, in which they describe a game where particles are randomly moved on a set of sites until no site contains more than one particle.
We will do it in a slightly more general case than the one described in \cite{CMS}, since we will not suppose that all partial topplings can occur.


\section{Informal model description}

In this section, we will recall some well known results about the abelian sandpile model (ASM), as described for instance in \cite{Redig}, thus giving motivation for the search of similar results in a more general model.


\subsection{The abelian sandpile model}
\subsubsection{Quick description}

The ASM is a Markov chain operating on a connected graph $G=(V\cup \{s\},E)$ where each vertex $v\in V$ has a positive counter $\eta(v)$, indicating how many sand grains are on it, and where $s$ is a special vertex called the sink, playing a different role described below. If the number of grains on a vertex $v\in V$ is smaller or equal than its number $d^G(v)$ of neighbors, $v$ is said to be \textit{stable}, and it is said to be \textit{unstable} otherwise. A \textit{configuration} $\eta$ of sand grains on $G$ is \textit{stable} if each vertex in $V$ is stable (i.e. $\eta(v)\leq d^G(v)$ for all $v\in V$), and \textit{unstable} otherwise. The state space for the ASM is the set of stable configurations.

At each step of the Markov chain, a sand grain is added at a random vertex $v\in V$ (according to a distribution $\mu$). If this vertex becomes unstable, it then \textit{topples}, giving a sand grain to each of its neighbors, which may in turn become unstable and topple. One of the vertices, called the \textit{sink}, and denoted $s$, behaves in a different way : it can absorb any number of sand grains and never topples. The topplings go on until we reach a new stable configuration (which will be our new state).
The recurrent states of the Markov Chain will be called \textit{recurrent configurations}.

\begin{exmp}
	An example of the ASM is given in Figure \ref{fig:asmex}. Only the recurrent configurations are being displayed here. The graph is a triangle with one of the three vertices being the sink, and sand grains can be added on both vertices in $V$.
	The transition matrix in the basis (\textcircled{1},\textcircled{2},\textcircled{3}) reads 
	$$\begin{pmatrix}
	0&\alpha&\beta\\
	\beta&0&\alpha\\
	\alpha&\beta&0	
	\end{pmatrix}.$$
	Note that $(1,1,1)$ is a left eigenvector with eigenvalue $\alpha+\beta=1$.
\end{exmp}

\begin{figure}[ht]
	\centering
	\includegraphics{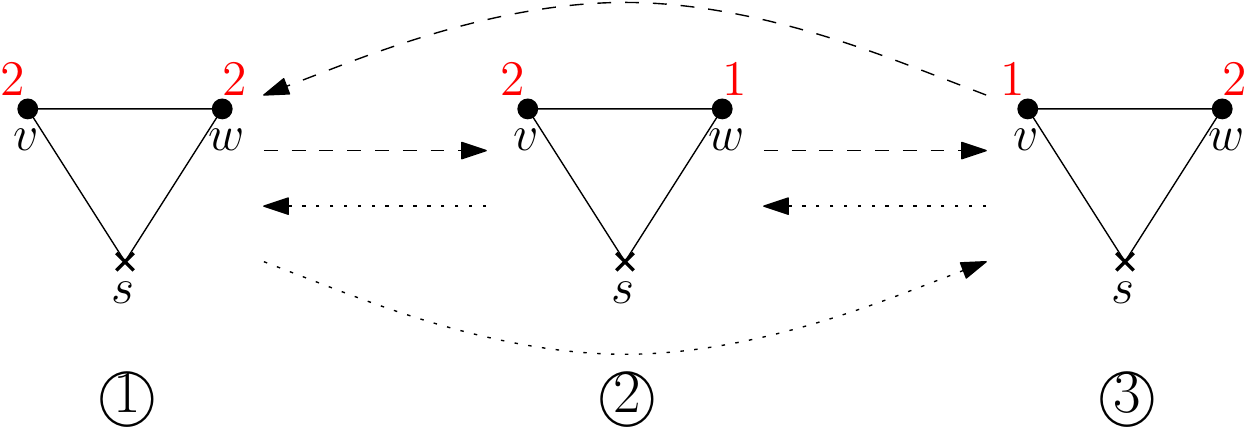}
	\caption{Dashed arrows correspond to transitions where we add a sand grain on $v$ (which happens with probability $\alpha$), and dotted arrows on $w$ (probability $\beta=1-\alpha$).}
	\label{fig:asmex}
\end{figure}

\subsubsection{Some results}

There are two important things to notice : first, the fact that the graph is connected ensures that the toppling process will end (if there are too many sand grains, some of them will necessary reach the sink), and second, the resulting stable state does not depend on the order in which vertices topple. Those two facts ensure that the Markov chain described above is well defined.

We talk about \textit{abelian} sandpile model because of the following result which can be found in \cite{Redig}.

\begin{prop}
	\label{prop:abelian}
	For $v\in V$, let $A_v$ be the operator ``add a sand grain on $v$ and make topplings until you reach a stable configuration''. 
	Then $\{A_v, v\in V\}$ is a commutative group acting on the set of recurrent configurations.
\end{prop}

As a corollary, we have the following property :

\begin{cor}
	The stationary distribution for the ASM  is the uniform distribution on recurrent configurations, and thus it does not depend on how we add sand grains on the graph (it does not depend on the distribution $\mu$).
\end{cor}


\subsection{The Stochastic Sandpile Model (SSM)}

\subsubsection{The Bernouilli Stochastic Sandpile Model (BSSM)}

This is the model described in \cite{BLS}, and is a generalization of the ASM. Here, randomness not only happens when sand grains are added onto the graph, but topplings are also random. More precisely, we fix $p\in [0,1]$, and when a vertex is unstable, each neighbor independently has a probability $p$ of receiving a grain from that vertex. The important thing about this model is that given an unstable vertex $v$ and a subset $D_v$ of $v$'s neighbors, there is a nonzero probability $D_v$ sends a sand grain exactly to the vertices of $D_v$. We can rephrase the last statement as follows : every partial toppling may occur.

\subsubsection{Quick description of the SSM}

The model described here and studied in the rest of the paper is a generalization of the BSSM in which only a subset of the partial topplings may occur.
The Markov chain is still operating on a graph $G=(V\cup{s},E)$ with sand grains, but this time for each $v\in V$, we have a probability distribution $\lambda_v$ on partial topplings of $v$ and a toppling value $M_v$ indicating how many sand grains are needed for $v$ to topple ($M_v$ will be the greatest number of sand grains $v$ can give away when toppling). Each time $v$ topples, it sends sand grains to only a subset of its neighbors, chosen at random according to $\lambda_v$.
One can notice that the graph structure becomes irrelevant, and we can instead consider the complete graph, or just the set of vertices, since all the needed information is contained in $(\lambda_v,v\in V)$.

\textit{Configurations} are defined in the same way as for the SSM, but for $\eta$ a configuration, a vertex $v$ will be said to be \textit{stable} if $\eta(v)$ is at most $M_v$, and \textit{unstable} otherwise.

Note that the Markov chain becomes a lot more complicated, since it is much harder to compute the transition probability from one stable state to another.
There is however a relatively simple way to do it by considering an extended Markov chain. This is explained in \cite{Selig}, will be detailed in Section \ref{sec:matrix}, and the following will be helpful.

\begin{exmp}
	\label{ex:ssmex}
	An example of the SSM is given in Figure \ref{fig:ssmex}. The graph is again a triangle with one of the three vertices being the sink. We consider that each vertex has 3 ways to topple, either by giving one sand grain to one of its two neighbors, or by giving one sand grain to both of them (with probabilities detailed on the figure).
	Here we represent all the unstable configurations appearing while making topplings, before finally reaching a stable configuration (stable configurations are the ones in the blue rectangle).
\end{exmp}

\begin{figure}[ht]
	\centering
	\includegraphics{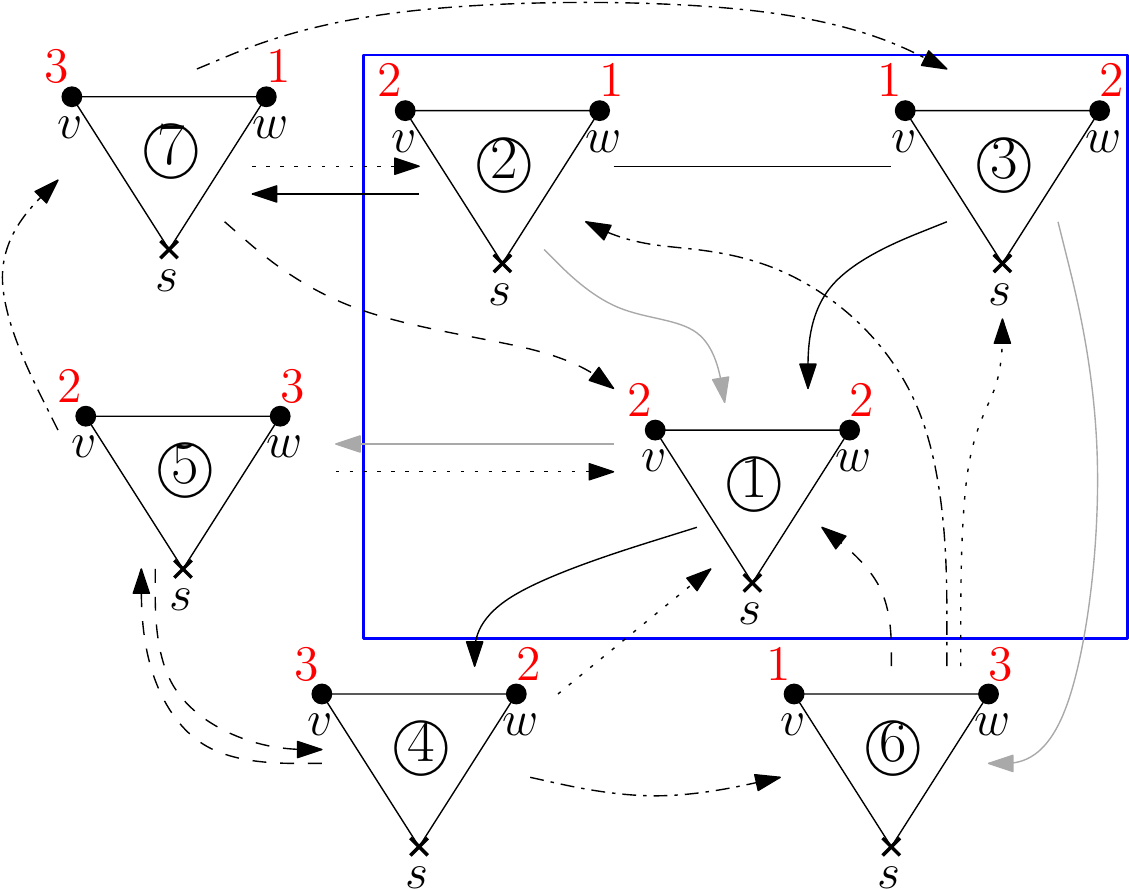}
	\caption{Solid black arrows correspond to transitions where a sand grain has just been added on $v$ (probability $a$), and solid gray arrows when a sand grain has been added on $w$ (probability $b=1-a$). Dotted arrows correspond to transitions where the unstable vertex gives a sand grain to the sink (which happens with probability $\alpha$), dashed arrows to transitions where it gives a sand grain to the other vertex (with probability $\beta$), and dashed-dotted arrows to transitions where it gives a sand grain to both (probability $\gamma=1-\alpha-\beta$). Stable configurations are the ones in the blue rectangle, the others are just intermediate unstable states.}
	\label{fig:ssmex}
\end{figure}


\subsubsection{Results to be investigated in the rest of the paper}

To ensure that the chain is well defined, we need the two following results :

\begin{prop}
	The toppling process almost surely ends.
\end{prop}

For this proposition to be true, we will need some additional hypothesis on the supports of the probability distributions $(\lambda_v,v\in V)$ on partial topplings.

\begin{prop}
	The order in which the vertices topple does not matter.
\end{prop}

Those two results will be needed to prove Conjecture \ref{Conj}.


\section{Model description.}
\label{sec:ssm}

We now give a more formal definition for the SSM.


\subsection{Notations and definitions.}
\label{sec:def}

Let $N$ be a nonnegative integer. The set $[N]:=\{1,\ldots,N\}$ represents the \textit{sites} (or the vertices, except the sink). A sandpile configuration on $[N]$ is a vector $\eta=(\eta_k,k\in[N])\in(\mathbb{N}-\{0\})^{[N]}$ : for $v$ in $[N]$, $\eta_v$ represents the number of sand grains on the site $v$.

For $v$ in $[N]$, we define $$\mathcal{S}_v=\{\nu\in\mathbb{Z}^{[N]}|\nu_v<0\text{, }\nu_{w}\geq 0\text{ for }w\neq v\text{, and }\sum_{i=1}^n \nu_i \leq 0\}$$
the set of possible topplings of $v$. If a toppling $\nu\in\mathcal{S}_v$ occurs, the state $\eta$ becomes $\eta+\nu $ : $v$ loses $-\nu_v$ sand grains, and any other site $w\neq v$ receives $\nu_{w}$ sand grains from $v$. If $\sum_{i=1}^n \nu_i <0$, it means some sand grains were send to the sink (Here the sink is not represented. Sending sand grains to is equivalent to losing sand grains).

For $v\in V$, let $\lambda_v$ be a probability distribution with finite support on $\mathcal{S}_v$ that describes the law of the random topplings of $v$.
We define $M_v=\max\{-\nu_v,\nu\in\mathcal{S}_v|\lambda_v(\nu)\neq 0\}$ to be the \textit{threshold} of site $v$, i.e.\  the highest number of sand grains $v$ can lose during a toppling.
If $\nu$ is a configuration, the site $v$ is said to be \textit{stable} if $\nu_v\leq M_v$, otherwise it is said to be \textit{unstable}.
The configuration $\nu$ is said to be stable if and only if all its sites are stable.


For $v\in[N]$, we also define $\delta^v$ to be the vector with a $1$ in $v^{th}$ position and zeros everywhere else : $\delta^v_i=\delta_{i,v}$ (where $\delta_{i,v}$ is the Kronecker symbol).

Finally, we need a probability distribution $\mu$ on $[N]$ which governs the way sand grains are added on the set of sites.


\subsection{State space and transitions}
\label{sec:stab}

The state space for our SSM is the set $St=[M_1]\times[M_2]\times\cdots\times[M_N]$ of stable configurations.

To properly describe the transitions, we need to define the random stabilization operator $\mathcal{RS}$ :
let $\eta$ be a sandpile configuration, we construct $\mathcal{RS}(\eta)$ by the following procedure :
\begin{enumerate}
\item If $\eta$ is stable, return $\mathcal{RS}(\eta)=\eta$ and exit
  the procedure.
\item Pick an arbitrary unstable site $v\in [N]$.
\item Draw a random toppling $D_v \in \mathcal{S}_v$ according to the
  probability distribution $\lambda_v$, independently from the past.
\item Replace $\eta$ by $\eta + D_v$.
\item Go to step 1.
\end{enumerate}
As we shall see, this procedure terminates almost surely under some
natural hypotheses on the $\lambda_v$'s, and the law of the resulting
$\mathcal{RS}(\eta)$ does not depend on which unstable site we pick at
each execution of the step 2.

Now the \emph{Stochastic Sandpile Model} is the Markov chain in which,
at each time step, we add a sand grain at a random site drawn
according to the probability distribution $\mu$, and then perform
random stabilization on the resulting configuration. For $\eta,\eta'$
two states, then the transition probability from $\eta$ to $\eta'$
reads
\begin{equation}
\label{eq:trans}
P_{\eta,\eta'}=\sum_{k=1}^N\mu(k)\mathbb{P}(\mathcal{RS}(\eta+\delta^k)=\eta').
\end{equation}

We now give a result about the finiteness of the stabilization process :

\begin{lem}
	\label{lem:finite}
	The two following assertions are equivalent :
	\begin{enumerate}
		\item[(i)] The random stabilization process is always almost surely finite regardless of $\mu$.
		\item[(ii)] For all nonempty subset $H$ of $[N]$, there exists a $\ell$ in $H$ such that 
		\begin{equation}
			\label{eq:fincon}
			\lambda_{\ell}(\{\nu\in\mathcal{S}_{\ell}|\sum_{k\in H} \nu_k<0\})>0.
		\end{equation}
	\end{enumerate}
\end{lem}

Note that we can rephrase condition $\eqref{eq:fincon}$ as follows :
For all nonempty subset $H$ of $[N]$, $H$ contains a vertex which has a nonzero probability to send sand grains outside of $H$ when toppling.

\begin{proof}
	
	Let us first prove that (i) implies (ii). Here we proceed by contraposition.
	Let us assume that there exists a nonempty subset $H\subset [N]$ from which sand grains can not escape. If the distribution $\mu$ allows sand grains to be added on some sites in $H$, then the total number of sand grains on $H$ will eventually become greater than $\sum_{k\in H} M_k$, ensuring that there is always at least one unstable site in $H$ and that the toppling process never ends.\\
	
	Let us now prove that (ii) implies (i).
	Let us denote $G_0=\emptyset$ and $H_0=[N]$.
	We recursively define $G_n,n\in \mathbb{N}$ and $H_n,n\in \mathbb{N}$ by
	$$G_{n+1}=\{\ell\in H_n|\lambda_{\ell}(\{\nu\in\mathcal{S}_{\ell}|\sum_{k\in H_n}\nu_k<0\})>0\},$$
	and $$H_{n+1}=H_n\backslash G_{n+1}.$$
	$G_1$ is the set of sites in $H_0$ having a nonzero probability to send sand grains to the sink (outside of $[N]$), and $H_1$ is its complementary.
	Similarly, $G_{n+1}$ is the set of sites in $H_n$ having a nonzero probability to send sand grains to $G_n$.
	The condition (ii) ensures that $\cup_{n\in\mathbb{N}} G_n=[N]$.
	We define the \textit{depth} of the Markov chain to be $D=\max\{n\in\mathbb{N}|G_n\neq \emptyset\}$.
	We will now proceed by induction on D.
	
	If $D=1$, it means that every site has a nonzero probability to send grains to the sink, and one easily checks that the stabilization occurs almost surely in that case, using the fact that the total number of sand grains on the graph will decrease during the stabilization process.
	
	Assume that the property is true for any case where $D\leq n$, which means that if $D\leq n$, the stabilization process always eventually ends. Consider the case $D=n+1$. Let us consider an unstable configuration. Each site in $G_1$ has a positive probability to send sand grains to the sink, and the total number of sand grains on the graph is finite, so the number of topplings in $G_1$ is almost surely finite. Furthermore, between two topplings in $G_1$, the number of topplings in $H_1$ is also almost surely finite (by induction hypothesis, since the process restricted to $H_1$ is of depth $n$). Therefore the number of topplings in the entire stabilization process is almost surely finite.
\end{proof}

From now on, we consider that the condition \eqref{eq:fincon} is satisfied.
This condition was used by Diaconis and Fulton in \cite{DF}, when they defined their \textit{growth model}.

We will now need to show that the result of the stabilization process does not depend on the choices we make when picking unstable sites to topple.

\subsection{The Abelian property}

We now prove the fact that the law of $\mathcal{RS}(\eta)$ is
independent of which unstable site is toppled at each step. This is
done by a coupling argument or ``deck of cards'' picture as in
\cite{DF}.

Let $(D_v^{(i)},i\geq 1,v\in [N])$ be a collection of independent random variables such that, for any $v\in[N]$ and $i\geq 1$, $D_v^{(i)}$ takes its value in $\mathcal{S}_v$ according to the probability distribution $\lambda_v$.
It is convenient to think of each site as being bound to an infinite deck of cards, where each card indicates a possible toppling at the site at hand. Each time we want to topple a vertex $v\in V$, we draw the top card of its deck, make the toppling according to the instructions on this card, and then throw the card away. For $i\geq 1$ and $v\in [N]$, $D_v^{(i)}$ represents the card that was originally at the $i^{th}$ position of the deck bound to the vertex $v$.

Given a realization of the $D_v^{(i)}$'s, the topplings are
deterministic. More precisely, each sandpile configuration $\eta$
shall be supplemented with a collection of \emph{counters} $c \in
\mathbb{N}^{[N]}$ keeping track of how many cards have been drawn
at each site. For $v \in [N]$, we define the toppling function
\begin{equation*}
  \begin{split}
    T_v: &\mathbb{Z}^N \times \mathbb{N}^N \to \mathbb{Z}^N \times \mathbb{N}^N \\
    &(\eta,c) \mapsto (\eta+D_v^{(c_v+1)},c+\delta^v)
  \end{split}
\end{equation*}
We perform the toppling $D_v^{(c_v+1)}$ at $v$ and increment the counter $c_v$ by one.

As one can easily check, for all $v,w\in[N]$, $T_v\circ T_w=T_w\circ T_v$.

We now define the determinized stabilization operator $\mathcal{DS}$:
for any sandpile configuration with counters $(\eta,c)$:
\begin{enumerate}
\item If $\eta$ is stable, return $\mathcal{DS}(\eta,c)=(\eta,c)$ and exit
  the procedure.
\item Pick an arbitrary unstable site $v\in [N]$.
\item Replace $(\eta,c)$ by $T_v(\eta,c)$.
\item Go to step 1.
\end{enumerate}
It looks like $\mathcal{DS}(\eta,c)$ depends on the sequence of
vertices $v_1,v_2,\ldots$ that we pick at each execution of the step
2. But it turns out that, if there exists one such sequence for which
the procedure terminates, then the procedure also terminates for any
other sequence and the resulting configuration is identical. To prove
this, we need to introduce some further terminology.





\begin{figure}[ht]
	\centering
	\includegraphics[width=1\linewidth]{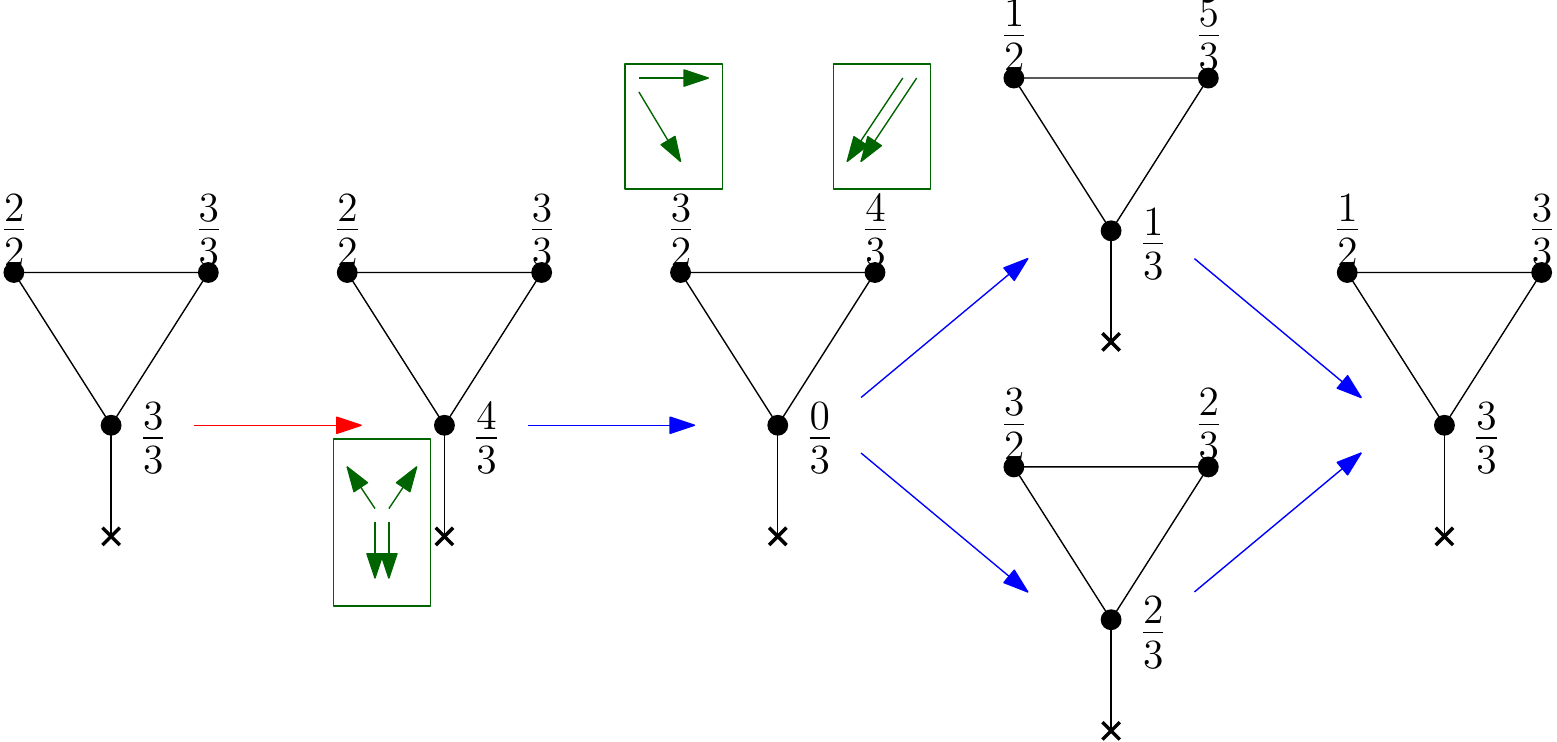}
	\caption{A detailed transition for the SSM on a triangle linked to the sink. The red arrow corresponds to the addition of a sand grain (on the lowest site here), and blue arrows to topplings. The cards on top of the decks have been represented in green. The above numbers near each site indicate how many sand grains are on this site, while the below number is the threshold of the site. Note that the resulting state does not depend on the order in which the vertices topple.}
	\label{fig:ssmstab}
\end{figure}

Given $(\eta,c)$ a sandpile configuration with counters, a sequence of
sites $(v_1,\ldots,v_n)$ is said to be \emph{legal} if $v_i$ is an unstable
vertex in $(T_{v_{i-1}} \circ T_{v_{i-2}} \circ \cdots
T_{v_1})(\eta,c)$ for all $i$ (in particular, $v_1$ must be an
unstable site in $\eta$). In other words, a legal sequence is a
possible choice of sites to topple in the first $n$ iterations of the
procedure $\mathcal{DS}(\eta,c)$. A legal sequence is said
\emph{stabilizing} if $(T_{v_{n}} \circ \cdots \circ
T_{v_1})(\eta,c)$ is stable.

 
\begin{lem}
  Let $(\eta,c)$ be a sandpile configuration with counters which
  admits a a stabilizing sequence $(v_1,\ldots,v_n)$, and let
  $(\xi,d)= (T_{v_{n}} \circ \cdots \circ T_{v_1})(\eta,c)$. Then the
  final counters $d$ are maximal in the sense that, for any legal
  sequence $(w_1,\ldots,w_p)$ for $(\eta,c)$, if $(\xi',d')=(T_{w_{p}}
  \circ \cdots \circ T_{w_1})(\eta,c)$ then we have $d'_v \leq d_v$
  for all $v \in [N]$.
 \end{lem}

 This lemma and its proof are inspired from \cite{Redig}.
 
 \begin{proof}
 	Let $(\xi,d)=T_{v_n}\circ T_{v_{n-1}}\circ\ldots\circ T_{v_1}(\eta,c)$ be the stable configuration we obtained. We have
 	$$(\xi,d)=(\eta+\sum_{v=1}^N \sum_{i=c_v+1}^{d_v} D_v^{(i)},d).$$
 	Suppose that the resulting state of a sequence of legal topplings starting from $(\eta,c)$ has counters $d'_v\leq d_v$ (this is always possible since we can choose $d'_v=c_v$ for all $v$), and for a site $\ell\in V$ an extra legal toppling can be performed. We put $$(\zeta,d')=(\eta+\sum_{v=1}^N \sum_{i=c_v+1}^{d'_v} D_v^{(i)},d').$$
 	Since an extra toppling can be performed at site $\ell$, we have $\zeta_{\ell}-\xi_{\ell}>0$ . This can be rewritten as 
 	$$\sum_{v=1}^N \sum_{i=d'_v+1}^{d_v} (D_v^{(i)})_{\ell}<0,$$
 	which gives us 
 	$$\sum_{i=d'_{\ell}+1}^{d_{\ell}} (D_{\ell}^{(i)})_{\ell}<-\sum_{v\neq \ell}\sum_{i=d'_v+1}^{d_v}(D_v^{(i)})_{\ell}\leq 0.$$
 	
 	The first sum being nonzero, it necessarily has at least one term, ensuring that $d'_{\ell}<d_{\ell}$ and thus $d'_{\ell}+1\leq d_{\ell}$ still holds once we make another toppling on $\ell$.
 \end{proof}
 
This result, along with Lemma \ref{lem:finite} proves that the stabilization operator is well defined.


\subsection{Some specific cases}
\label{sec:spec}

\begin{enumerate}
	\item If for all $i\geq 1$, for all $k\in[N]$, we always have $(D_k^{(i)})_k=-1$, this corresponds to the case where a toppling vertex will always give exactly one sand grain to either an other vertex or the sink. It is easy to see that in this case, the only possible recurrent state is the one with exactly one sand grain on every site : Each time we add a sand grain, it will move from site to site until it falls into the sink.
	
	\item For all $k,\ell$ in $[N]$, let $d_{k,\ell}=d_{\ell,k}$ and be fixed positive integers.
	Let $p\in]0,1]$ be a fixed real number.
	
	For $i\geq 1$, $k,\ell$ in $[N]$ and $j$ in $\{1,\ldots,d_{k,\ell}\}$, $B^{(i)}_{k,\ell}(j)$ are independent Bernoulli random variables with parameter p.
	
	The case where for all $i\geq 1$ and for all $k\neq \ell$, we have
	$$(D_k^{(i)})_{\ell}=\sum_{j=1}^{d_{k,\ell}}B^{(i)}_{k,\ell}(j)$$
	and
	$$(D_k^{(i)})_k=-\Big(\sum_{m\neq k}(D_k^{(i)})_m+\sum_{j=1}^{d_{k,k}}B^{(i)}_{k,k}(j)\Big)$$
	corresponds to the Stochastic Sandpile Model on a graph defined by Chan, Marckert and Selig in \cite{CMS} : indeed, for $k\neq\ell$, $d_{k,\ell}$ is the number of edges between the vertices $k$ and $\ell$, and $d_{k,k}$ is the number of edges between $k$ and the sink.
	
	\item The previous case with $p$ fixed to $1$ corresponds to the abelian sandpile model on a graph.

\end{enumerate}


\section{Around the stationary distribution}
\label{sec:stat}

We now show how to precisely compute the stationary distribution, and we then show Conjecture \ref{Conj}.

\subsection{The transition matrix}
\label{sec:matrix}

We here quickly explain how to compute the transition matrix on simple examples. We follow the same procedure as in \cite[Section 5.7.1.1]{Selig}, by defining an extended Markov chain.

We first extend the state space : the new state space is the set $St'=(\mathbb{N}-\{0\})^{[N]}$ of sandpile configurations. 

The transitions are defined as follows :
Let $\eta$, $\eta'$ be two sandpile configurations, the transition probability from $\eta$ to $\eta'$ reads :

\begin{equation}
\label{eq:enrtrans}
\tilde{P}_{\eta,\eta'}=
\begin{cases}
\mu(k)\text{ if $\eta$ is stable and $\eta'=\eta+\delta^k$ for some $k\in[N]$,}\\
\mathbb{P}(\overline{T}_k(\eta)=\eta')\text{ if $\eta$ is unstable and $k$ is the smallest unstable site in $\eta$,}\\
0\text{ otherwise.}
\end{cases}
\end{equation}

This Markov chain is much simpler, because only one operation (be it a toppling or the addition of a sand grain) will happen at each transition.

The transition matrix for this Markov chain reads $\tilde{P}=\begin{pmatrix}
A&B\\C&D
\end{pmatrix}$
where $A=(\tilde{P}_{\eta,\eta'}|\eta,\eta'\in St)$, $B=(\tilde{P}_{\eta,\eta'}|(\eta,\eta')\in St\times St')$,
$C=(\tilde{P}_{\eta,\eta'}|(\eta,\eta')\in St'\times St)$ and $D=(\tilde{P}_{\eta,\eta'}|\eta,\eta'\in St')$.

We are now able to describe the transition matrix $P$ for the SSM. Indeed, for all $\eta,\eta'\in St$, we have
$$P_{\eta,\eta'}=\tilde{P}_{\eta,\eta'}+\sum_{\eta_1,\eta_2,\ldots,\eta_k\in St'}\tilde{P}_{\eta,\eta_1}\tilde{P}_{\eta_1,\eta_2}\ldots\tilde{P}_{\eta_{k-1}},\eta_k\tilde{P}_{\eta_k,\eta'}$$
which can be written as : 
$$P=A+B\Big(\sum_{n\geq 0} D^n\Big)C,$$
or more conveniently
\begin{equation}
\label{eq:relat}
P=A+B(I-D)^{-1}C.
\end{equation}

\begin{exmp}
	The transition matrix of the extended Markov chain represented in Figure \ref{fig:ssmex} in the basis $\{\text{\textcircled{1}},\ldots,\text{\textcircled{7}}\}$ reads :
	\begin{equation}
	\begin{pmatrix}
	0&0&0&a&b&0&0\\
	b&0&0&0&0&0&a\\
	a&0&0&0&0&b&0\\
	\alpha&0&0&0&\beta&\gamma&0\\
	\alpha&0&0&\beta&0&0&\gamma\\
	\beta&\gamma&\alpha&0&0&0&0\\
	\beta&\alpha&\gamma&0&0&0&0
	\end{pmatrix}
	\end{equation}
And if we write
$A=\begin{pmatrix}
0&0&0\\
b&0&0\\ 
a&0&0
\end{pmatrix}$,
$B=\begin{pmatrix}
a&b&0&0\\
0&0&0&a\\
0&0&b&0
\end{pmatrix}$,
$C=\begin{pmatrix}
\alpha&0&0\\
\alpha&0&0\\
\beta&\gamma&\alpha\\
\beta&\alpha&\gamma
\end{pmatrix}$
and
$D=\begin{pmatrix}
0&\beta&\gamma&0\\
\beta&0&0&\gamma\\
0&0&0&0\\
0&0&0&0
\end{pmatrix}$
and by subsiding those matrices in Equation \eqref{eq:relat}, we find the transition matrix for the SSM in the basis $\{\text{\textcircled{1},\textcircled{2},\textcircled{3}}\}$ :

$$\begin{pmatrix}
\frac{\alpha+\beta\gamma}{1-\beta}&\frac{\gamma(a(\gamma+\alpha\beta)+b(\alpha+\beta\gamma))}{1-\beta^2}&\frac{\gamma(a(\alpha+\beta\gamma)+b(\gamma+\alpha\beta))}{1-\beta^2}\\
b+a\beta&a\alpha&a\gamma\\
a+b\beta&b\gamma&b\alpha
\end{pmatrix}.$$
Note that $(\frac{\beta+1}{\gamma},1,1)$ is a left eigenvector for that matrix.
The fact that this vector does not depend on $(a,b)$ is quite remarkable, and we will now show that this fact is true in general.
\end{exmp}

\subsection{Commuting matrices and stationary distribution.}
\label{sec:thm}

We have explained how to compute the transition matrix $P$ for the SSM. This matrix of course depends on the distribution $\mu$ governing the additions of sand grains on the sites, but our aim is to show that the stationary distribution does not.
For $k$ in $[N]$, we will note $P^{(k)}$ the matrix in $St$ with coefficients reading :
\begin{equation}
\label{eq:spe}
P^{(k)}_{\eta,\eta'}=\mathbb{P}(\mathcal{RS}(\eta+\delta^k)=\eta')
\end{equation}
The rows of this matrix give the probability to go from one state to another by adding a sand grain on $k$ and then making topplings until we reach a stable state.

\begin{prop}
	For all $k,\ell\in[N]$, we have $P^{(k)}P^{(\ell)}=P^{(\ell)}P^{(k)}$.
\end{prop}

\begin{proof}
Let $(\eta,c)$ be a stable state sandpile configuration with counters, and let $k,\ell$ be in $[N]$.
If $k=\ell$, then $P^{(k)}P^{(\ell)}=P^{(\ell)}P^{(k)}$ and the result holds.

If $k\neq \ell$, let $v_1,\ldots,v_n$ be a stabilizing sequence for $(\eta+\delta^{(k)},c)$ : we note $(\xi,d)=\mathcal{DS}(\eta+\delta^{(k)},c)=T_{v_n}\circ \cdots \circ T_{v_1}(\eta+\delta^{(k)},c)$.
Let $w_1,\ldots,w_m$ be a stabilizing sequence for $(\xi+\delta^{(\ell)},d)$.
Then $v_1,\ldots,v_n,w_1,\ldots,w_m$ is a stabilizing sequence for $(\eta+\delta^{(k)}+\delta^{(\ell)},c)$.

Thus we have $\mathcal{DS}(\mathcal{DS}(\eta+\delta^{(k)},c)+(\delta^{(\ell)},0))=\mathcal{DS}(\eta+\delta^{(k)}+\delta^{(\ell)},c)$. The same argument can be used after switching the roles of $k$ and $\ell$, which shows that $\mathcal{DS}(\mathcal{DS}(\eta+\delta^{(k)},c)+(\delta^{(\ell)},0))=\mathcal{DS}(\mathcal{DS}(\eta+\delta^{(\ell)},c)+(\delta^{(k)},0))$, allowing to conclude that $P^{(k)}P^{(\ell)}=P^{(\ell)}P^{(k)}$.
\end{proof}

\begin{prop}
	The stationary distribution for the SSM does not depend on $\mu$.
\end{prop}

\begin{proof}
	Equation \eqref{eq:trans} allows us to write $P=\sum_{k\in[N]}\mu(k)P^{(k)}$.
	Let us now consider $P'=\sum_{k\in[N]}\mu'(k)P^{(k)}$ with $\mu'$ an other probability distribution on $[N]$.
	Since for all $k,\ell$ in $[N]$, $P^{(k)}P^{(\ell)}=P^{(\ell)}P^{(k)}$, we can conclude that $PP'=P'P$.
	Thus, if $\pi$ is the stationary distribution of the Markov chain when sand grains are added according to $\mu$, we have $\pi PP'=\pi P'=\pi P'P$ which shows that $\pi P'$ is an eigenvector with eigenvalue $1$ for $P$, and since $P'$ is a probability matrix and $1$ is a non degenerated eigenvalue for $P$, we necessarily have $P'\pi=\pi$.
\end{proof}

\begin{rem}
Note that we do not know necessary and sufficient conditions for the chain to be irreducible and aperiodic.

However, if $\mu$ is strictly positive, the Markov chain will be irreducible on the recurrent class containing the maximal state $(M_1,M_2,\ldots,M_N)$.

In \cite{CMS}, the SSM is defined on a connected graph with a sink, and topplings are such that any partial toppling can happen. The sand grain addition law $\mu$ is also strictly positive (any site can receive a sand grain). Those conditions ensure that the SSM is irreducible on the recurrent class containing the maximal configuration $(M_1,\ldots,M_N)$. The authors of \cite{CMS} also were able to describe the set of recurrent states in this case.

An example of a case where the SSM is not aperiodic is the following : if for all $k\in V$, and for all $\nu\in\mathcal{S}_k$, $\sum_{i\in V} \nu_i$ is even, then the process will alternate between states with an even number of sand grains and states with an odd number of them. This of course works with any similar condition (if the GCD of all possible values for $\sum_{i\in V} \nu_i$ is not $1$ for example).
\end{rem}

\section*{Acknowledgments}
The author would like to thank Andrea Sportiello for sharing his vision and pertinent questions, and J\'er\'emie Bouttier for his help.

\bibliographystyle{plain}
\bibliography{ssmbib}

\end{document}